\documentclass[12pt]{amsart}

\usepackage{amsmath,amsthm,amssymb,amsfonts}
\usepackage[shortlabels]{enumitem}
\usepackage[colorlinks=true,citecolor=red,linkcolor=blue]{hyperref}
\usepackage[left=2.7cm,right=2.7cm,top=2.5cm,bottom=2.5cm]{geometry}

\theoremstyle{plain}
\newtheorem{thm}{Theorem}[section]
\newtheorem{prop}[thm]{Proposition}
\newtheorem{lemma}[thm]{Lemma}
\newtheorem{cor}[thm]{Corollary}

\theoremstyle{remark}
\newtheorem{remark}[thm]{Remark}

\newcommand{\N}{\mathbb N}
\newcommand{\Z}{\mathbb Z}
\newcommand{\C}{\mathbb C}
\newcommand{\norm}[1]{\left\|#1\right\|}
\newcommand{\abs}[1]{\left|#1\right|}

\begin{document}

\title[Positive strongly Kreiss bounded operators]
{Positive strongly Kreiss bounded operators on $L^p$-spaces have logarithmic power growth}

\author[L. Arnold]{Loris Arnold}
	
	\address[L. Arnold]{Normandie Univ, UNICAEN, CNRS, LMNO, 14000 Caen, France}
	\email{lfj.arld@gmail.com}


\begin{abstract}
Let $1<p<\infty$ and let $T$ be a positive strongly Kreiss bounded
operator on an $L^p$-space. We prove that there exist constants
$C,\kappa>0$ such that
\[
\norm{T^N}\leq C\bigl(\log(N+1)\bigr)^\kappa,
\qquad N\geq1.
\]
This answers a question raised by Arnold and Cuny concerning positive
strongly Kreiss bounded operators on $L^p$-spaces. It also shows that
the polynomial exponent appearing in the corresponding problem of
Deng, Lorist and Veraar has infimum zero. The proof is elementary:
positivity yields a local $\ell^p$-estimate for the orbit on intervals
of length of order $\sqrt N$; a positive vector-valued lifting of $T$
then produces a self-improving estimate which can be iterated down to
the logarithmic scale.
\end{abstract}

\maketitle


\section{Introduction}

Let $X$ be a Banach space and let $T\in B(X)$. Recall that $T$ is
\emph{strongly Kreiss bounded} if there exists a constant $C>0$ such
that
\[
\norm{(\lambda I-T)^{-m}}
\leq
\frac{C}{(\abs{\lambda}-1)^m},
\qquad
\abs{\lambda}>1,\quad m\geq1.
\]
We denote by $C_{SK}(T)$ the infimum of all such constants $C > 0$. Equivalently, by a result of Nevanlinna, 
\begin{equation}\label{eqSK}
\norm{e^{zT}}
\leq
C_{SK}(T)e^{\abs z},
\qquad z\in\C.
\end{equation}
We shall use \eqref{eqSK} throughout the paper.

The growth of the powers of strongly Kreiss bounded operators has been
studied extensively. On arbitrary Banach spaces, Lubich and Nevanlinna
proved the general estimate
\[
\norm{T^N}=O(\sqrt N),
\]
and showed that this order of growth is optimal; see
\cite{LubichNevanlinna}. On Hilbert spaces the situation is radically
different. Cohen, Cuny, Eisner and Lin proved that every strongly
Kreiss bounded operator satisfies
\begin{equation*}
\norm{T^N}
=
O\bigl((\log(N+1))^\kappa\bigr)
\end{equation*}
for some $\kappa>0$; see \cite[Theorem 4.5]{CCEL}.

More recently, Cuny and the author obtained a sharp result on arbitrary
$L^p$-spaces (see \cite{ArnoldCuny}). If $1<p<\infty$ and $T$ is strongly Kreiss bounded on
$L^p$, then
\begin{equation*}
\norm{T^N}
\leq
C N^{\left|\frac12-\frac1p\right|}
\bigl(\log(N+1)\bigr)^\kappa.
\end{equation*}
Moreover, the polynomial exponent
$\left|\frac12-\frac1p\right|$ is optimal, see \cite[Proposition 1.2.]{ArnoldCuny}. In the same paper, it was also shown that positive strongly Kreiss bounded operators on AL-spaces and AM-spaces satisfy a logarithmic power growth estimate of the form $\norm{T^N} =O((\log(N+1))^{\kappa})$.

The positivity assumption leads to substantially better estimates.
Cuny and the author proved that if $T$ is positive and strongly Kreiss
bounded on $L^p$, then
\[
\norm{T^N}
\leq
C N^{1/\overline p}\bigl(\log(N+1)\bigr)^\kappa,
\qquad
\overline p:=\max\{p,p'\},
\]
(see \cite[Proposition 5.2.]{ArnoldCuny}) and asked whether the remaining polynomial factor could be removed
altogether; see \cite[Question 5.8]{ArnoldCuny}. More precisely, they
asked whether positivity implies an estimate of the form
\[
\norm{T^N}
\leq
C\bigl(\log(N+1)\bigr)^\kappa.
\]

Deng, Lorist and Veraar subsequently developed a general geometric
framework for strongly Kreiss bounded operators on UMD spaces. Among
other results, they obtained improved estimates for positive operators
on $p$-convex and $q$-concave Banach lattices; see
\cite[Theorem 4.6]{DLV}. In the particular case of positive operators
on $L^p$, they explicitly asked for the smallest polynomial exponent
which can occur; see \cite[Problem 5.6]{DLV}.

The aim of this note is to answer the question of Cuny and the author
affirmatively.

\begin{thm}\label{thmMain}
Let $1<p<\infty$, let $(\Omega,\mu)$ be a $\sigma$-finite measure
space and let
$T \in B(L^p(\Omega))$
be a positive strongly Kreiss bounded operator. Then there exist
constants $C,\kappa>0$, depending only on $p$ and $C_{SK}(T)$, such that
\begin{equation*}
\norm{T^N}_{B(L^p(\Omega))}
\leq
C\bigl(\log(N+1)\bigr)^\kappa,
\qquad N\geq1.
\end{equation*}
\end{thm}

It is worth stressing that logarithmic growth cannot in general be
replaced by power boundedness. Indeed, the positive weighted backward
shifts considered in \cite{CCEL} (see Proposition 4.9.) provide, for every
$\gamma>0$, positive strongly Kreiss bounded operators satisfying
\[
\norm{T^N}\asymp(\log(N+1))^\gamma.
\]


\section{Preliminary estimates}

We begin by recalling the form of the Krivine calculus which will be
used below.

\begin{lemma}[Krivine calculus]\label{lemKrivine}
Let $X$ be a Banach lattice, let $S:X\to X$ be positive and let
$1\leq r<\infty$. Then, for every finite family
$x_1,\ldots,x_m\in X$,
\begin{equation*}
\left(
\sum_{j=1}^m \abs{Sx_j}^r
\right)^{1/r}
\leq
S\left(
\left(
\sum_{j=1}^m\abs{x_j}^r
\right)^{1/r}
\right).
\end{equation*}
\end{lemma}

This is a standard consequence of the Krivine functional calculus;
see \cite[Theorem 1.d.1]{LT}.

We shall also use the following elementary consequence of Stirling's
formula (see for example \cite[Lemma 3.4]{CCEL}).

\begin{lemma}\label{lemStirling}
There exist constants $c_{exp}>0$ and $N_0\geq1$ such that, for every
integer $N\geq N_0$ and every integer $k$ satisfying
\[
N-4\sqrt N\leq k\leq N,
\]
one has
\begin{equation*}
\frac{N^k}{k!}
\geq
c_{exp}\frac{e^N}{\sqrt N}.
\end{equation*}
\end{lemma}

Our first estimate was already observed in
\cite[Lemma 5.3]{ArnoldCuny}. We include the short argument in the
form needed below.

\begin{lemma}\label{lemInitialLocal}
Let $1<p<\infty$ and let $T$ be a positive strongly Kreiss bounded
operator on $L^p(\Omega)$. Then there exists
$C_0>0$, depending only on $p$ and $C_{SK}(T)$, such that
\begin{equation*}
\sum_{N+2-2\sqrt N\leq n\leq N}
\norm{T^n x}_p^p
\leq
C_0 N^{p/2}\norm{x}_p^p
\end{equation*}
for all $N\geq1$ and $x\in L^p(\Omega)$.
\end{lemma}

\begin{proof}
By positivity, assume $x\geq 0$. Lemma \ref{lemStirling} yields $N^n/n! \geq c_{exp} e^N/\sqrt{N}$ for $N+2-2\sqrt{N}\leq n\leq N$ and $N$ sufficiently large. Since $x\geq 0$,
\[
|e^{NT}x|^p 
= \left( \sum_{n=0}^\infty \frac{N^n}{n!}T^n x \right)^p 
\geq \frac{c_{exp}^p e^{pN}}{N^{p/2}} \sum_{N+2-2\sqrt{N}\leq n\leq N} |T^n x|^p.
\]
Integrating and applying \eqref{eqSK} gives the desired result.
\end{proof}

The next elementary lemma turns estimates on terminal
$\sqrt N$-windows into global orbit estimates.

\begin{lemma}\label{lemLocalGlobal}
Let $\alpha\geq0$ and $A>0$. Assume that
\begin{equation}\label{eqLocalGeneral}
\sum_{N+2-2\sqrt N\leq n\leq N}
\norm{T^n x}_p^p
\leq
A N^\alpha\norm{x}_p^p
\end{equation}
for every $N\geq1$ and every $x\in L^p(\Omega)$. Then
\begin{equation}\label{eqGlobalGeneral}
\sum_{n=1}^N\norm{T^n x}_p^p
\leq
2^{2\alpha+1} A
N^{\alpha+1/2}\norm{x}_p^p
\end{equation}
for every $N\geq1$ and every $x\in L^p(\Omega)$.
\end{lemma}

\begin{proof}
Let $M\geq1$. For every $1\leq k\leq M$, applying
\eqref{eqLocalGeneral} with $N=k^2$ yields
\[
\sum_{(k-1)^2+1\leq n\leq k^2}
\norm{T^n x}_p^p
\leq
A k^{2\alpha}\norm{x}_p^p,
\]
since $k^2+2-2k=(k-1)^2+1.$ Therefore
\[
\sum_{n=1}^{M^2}\norm{T^n x}_p^p
\leq
A\sum_{k=1}^M k^{2\alpha}\norm{x}_p^p
\leq
A M^{2\alpha+1}\norm{x}_p^p.
\]
For arbitrary $N$, choose $M=\lceil\sqrt N\rceil$. Then
$N\leq M^2$ and $M\leq2\sqrt N$, which gives
\eqref{eqGlobalGeneral}.
\end{proof}

Combining Lemmas \ref{lemInitialLocal} and
\ref{lemLocalGlobal} gives the first global estimate.

\begin{cor}\label{corInitialGlobal}
There exists $C>0$ such that
for every $N\geq1$ and every $x\in L^p(\Omega)$,
\begin{equation*}
\sum_{n=1}^N\norm{T^n x}_p^p
\leq
CN^{(p+1)/2}\norm{x}_p^p.
\end{equation*}
\end{cor}

\section{The proof of the main Theorem}
\subsection{A positive vector-valued lifting}

Let $\mathcal X := L^p(\Omega;\ell^p(\Z)) \simeq \ell^p(\Z;L^p(\Omega))$ and let $U$ be the bilateral right shift $(Ua)_j=a_{j-1}$ on $\ell^p(\Z)$. Define $\mathbf T := U\otimes T$, so that $(\mathbf TF)_j = T(f_{j-1})$ for $F=(f_j)_{j\in\Z}\in\mathcal X$. Setting $\abs{F}_{\ell^p} := \left(\sum_{j\in\Z}\abs{f_j}^p\right)^{1/p}$, we have
\begin{equation}\label{eqFubini}
\norm{\abs{F}_{\ell^p}}_{L^p} = \norm{F}_{\mathcal X}.
\end{equation}
The next lemma shows that $\mathbf T$ is strongly Kreiss bounded with $C_{SK}(\mathbf{T}) \le C_{SK}(T)$.
\begin{lemma}\label{lemLifting}
If $T$ is strongly Kreiss bounded, then $\mathbf T$ satisfies
\begin{equation*}
\norm{e^{z\mathbf T}}_{B(\mathcal X)} \leq C_{SK}(T)e^{\abs z}, \qquad z\in\C.
\end{equation*}
\end{lemma}

\begin{proof}
By Lemma \ref{lemKrivine}, for every $k\geq0$,
\[
\abs{(U^k\otimes T^k)F}_{\ell^p} 
= \left(\sum_{j\in\Z}\abs{T^k f_{j-k}}^p\right)^{1/p} \leq T^k\abs{F}_{\ell^p}.
\]
Thus,
\[
\abs{e^{z\mathbf T}F}_{\ell^p} 
\leq \sum_{k=0}^\infty \frac{\abs z^k}{k!} \abs{(U^k\otimes T^k)F}_{\ell^p} 
\leq \sum_{k=0}^\infty \frac{\abs z^k}{k!} T^k\abs{F}_{\ell^p} 
= e^{\abs zT}\abs{F}_{\ell^p}.
\]
Taking the $L^p$-norm and applying \eqref{eqFubini} and \eqref{eqSK} yields $\norm{e^{z\mathbf T}F}_{\mathcal X} \leq C_{SK}(T)e^{\abs z}\norm{F}_{\mathcal X}$.
\end{proof}


\subsection{The bootstrap argument}

The following is the key self-improvement step.

\begin{lemma}\label{lemBootstrap}
Suppose that $p>\beta>1$ and that there exists $A>0$ such that for every $N\geq1$ and every $x\in L^p(\Omega)$
\begin{equation}\label{eqHbeta}
\sum_{n=1}^N\norm{T^n x}_p^p
\leq
A N^\beta\norm{x}_p^p.
\end{equation}
Then there exists a constant $D_{p}>0$, depending only on
$p$ and $C_{SK}(T)$, such that for every $N\geq1$ and every $x\in L^p(\Omega)$
\begin{equation*}
\sum_{n=1}^N\norm{T^n x}_p^p
\leq
D_{p}A
N^{(\beta+1)/2}\norm{x}_p^p.
\end{equation*}
\end{lemma}

\begin{proof}
Let $x \in L^p(\Omega)$ and fix $N$ sufficiently large, set $m:=\left\lceil4\sqrt N\right\rceil$, and define $Y=(Y_j)_{j\in\Z}\in\mathcal X$ by
\[
Y_j
=
\begin{cases}
T^j x,&1\leq j\leq m,\\
0,&\text{otherwise}.
\end{cases}
\]
By \eqref{eqHbeta},
\begin{equation*}
\norm{Y}_{\mathcal X}^p = \sum_{j=1}^m\norm{T^j x}_p^p \leq A m^\beta\norm{x}_p^p \le 5^{p} A N^{\beta/2}\norm{x}_p^p.
\end{equation*}
Let $W:=e^{N\mathbf T}Y$. By Lemma \ref{lemLifting},
\begin{equation}\label{eqWUpper}
\norm{W}_{\mathcal X}^p \le 5^{p}C_{SK}(T)^p A e^{pN}N^{\beta/2}\norm{x}_p^p.
\end{equation}
For an integer $n$ in the interval $N+2-2\sqrt N\leq n\leq N$, we have $W_n = \sum_{k\geq0} \frac{N^k}{k!}T^kY_{n-k}$. Whenever $1\leq n-k\leq m$, $T^kY_{n-k}=T^n x$, and therefore
\begin{equation}\label{eqCoordinateW}
W_n = T^n x \left(\sum_{\substack{k\geq0\\1\leq n-k\leq m}} \frac{N^k}{k!}\right).
\end{equation}
Let $K_N := \left\{ k\in\N_0: N-4\sqrt N\leq k\leq N-2\sqrt N \right\}$. For $N$ sufficiently large, if $N+2-2\sqrt N\leq n\leq N$ and $k\in K_N$, then $1\leq n-k\leq m$. Moreover, $|K_N| \ge \sqrt N$, and by Lemma \ref{lemStirling}, $\frac{N^k}{k!} \geq c_{exp}\frac{e^N}{\sqrt N}$ for $k\in K_N$. Consequently,
\[
\sum_{\substack{k\geq0\\1\leq n-k\leq m}} \frac{N^k}{k!} \geq \sum_{k\in K_N}\frac{N^k}{k!} \geq c_{exp}e^N.
\]
It follows from \eqref{eqCoordinateW} that $\norm{W_n}_p \geq c_{exp}e^N\norm{T^n x}_p$. Hence
\begin{equation}\label{eqWLower}
\norm{W}_{\mathcal X}^p \geq c_{exp}^pe^{pN} \sum_{N+2-2\sqrt N\leq n\leq N} \norm{T^n x}_p^p.
\end{equation}
Combining \eqref{eqWUpper} and \eqref{eqWLower} yields $\sum_{N+2-2\sqrt N\leq n\leq N} \norm{T^n x}_p^p \le 5^p\left( \frac{C_{SK}(T)}{c_{exp}}\right)^p A N^{\beta/2}\norm{x}_p^p$. Applying Lemma \ref{lemLocalGlobal} with $\alpha=\beta/2$ gives
\[
\sum_{n=1}^N\norm{T^n x}_p^p \le D_p A N^{(\beta+1)/2}\norm{x}_p^p,
\]
with $D_p := 2^{p+1}5^p\left( \frac{C_{SK}(T)}{c_{exp}}\right)^p$.
\end{proof}

Repeated application of the preceding lemma reduces the exponent to $1$ at the cost of an additional logarithmic factor.

\begin{prop}\label{propOrbit}
There exist constants $C_p,\kappa_p>0$, depending only on $p$ and
$C_{SK}(T)$, such that
\begin{equation}\label{eqOrbitFinal}
\sum_{n=1}^N\norm{T^n x}_p^p
\leq
C_pN\bigl(\log(N+2)\bigr)^{\kappa_p}
\norm{x}_p^p
\end{equation}
for every $N\geq1$ and $x\in L^p(\Omega)$.
\end{prop}

\begin{proof}
Let $N\ge 2$ and set $\beta_0:=\frac{p+1}{2}$ and $\beta_{j+1}:=\frac{\beta_j+1}{2}$, so that $\beta_j = 1+\frac{p-1}{2^{j+1}}$. Corollary \ref{corInitialGlobal} gives the desired estimate for $j=0$. Inspecting the constants in Lemma \ref{lemBootstrap} and Lemma \ref{lemLocalGlobal} yields the existence of $C>0$ such that
\begin{equation}\label{eqIteration}
\sum_{n=1}^N\norm{T^n x}_p^p \leq C D_p^j N^{1+\frac{p-1}{2^{j+1}}}\norm{x}_p^p
\end{equation}
for every $j\geq0$ and $x \in L^p(\Omega)$. Choosing $j = \lfloor\log_2\log(N+1)\rfloor$ gives $2^j\leq\log(N+1)<2^{j+1}$, from which $N^{\frac{p-1}{2^{j+1}}} \leq e^{p-1}$ and $D_p^j \leq (\log(N+1))^{\log_2 D_p}$. Inserting these bounds into \eqref{eqIteration} yields \eqref{eqOrbitFinal}. 
\end{proof}


\subsection{Conclusion}

Since $T$ is positive and strongly Kreiss bounded on $L^p(\Omega)$, its adjoint $T^*$ on $L^{p'}(\Omega)$ is also positive and satisfies $\|e^{zT^*}\| = \|e^{zT}\| \leq C_{SK}(T)e^{|z|}$.
Applying Proposition \ref{propOrbit} to $T$ and $T^*$ yields constants $C_p, C_{p'}, \kappa_p, \kappa_{p'} > 0$ such that, for all $N \ge 1$,
\[
\sum_{n=1}^N \|T^n x\|_p^p \leq C_p N (\log(N+1))^{\kappa_p} \|x\|_p^p
\quad \text{and} \quad
\sum_{n=1}^N \|(T^*)^n y\|_{p'}^{p'} \leq C_{p'} N (\log(N+1))^{\kappa_{p'}} \|y\|_{p'}^{p'}.
\]
Fix $N \geq 2$ and unit vectors $x \in L^p(\Omega)$, $y \in L^{p'}(\Omega)$. Using $\langle T^N x, y \rangle = \langle T^k x, (T^*)^{N-k} y \rangle$ for $1 \leq k \leq N-1$ together with Hölder's inequality, we get
\begin{align*}
(N-1)|\langle T^N x, y \rangle|
&\leq \sum_{k=1}^{N-1} \|T^k x\|_p \|(T^*)^{N-k} y\|_{p'}
\leq \left( \sum_{k=1}^{N-1} \|T^k x\|_p^p \right)^{1/p} \left( \sum_{k=1}^{N-1} \|(T^*)^k y\|_{p'}^{p'} \right)^{1/p'} \\
&\leq C_p^{1/p} C_{p'}^{1/p'} N (\log(N+1))^{\frac{\kappa_p}{p} + \frac{\kappa_{p'}}{p'}}.
\end{align*}
Dividing by $N-1$ and taking the supremum over unit vectors gives the existence of $C>0$ such that for every $N\ge1$,
\[
\|T^N\| \leq C (\log(N+2))^\kappa
\]
with $\kappa = \frac{\kappa_p}{p} + \frac{\kappa_{p'}}{p'}.$


\section{Remarks}

\begin{remark}[Relation with the Banach lattice results of Deng--Lorist--Veraar]
The argument above is strongly lattice-theoretic. In particular,
Lemma \ref{lemKrivine} is valid on arbitrary Banach lattices and
does not depend on the particular representation of $L^p$.

The specifically $L^p$ ingredient is the exact identity
\[
\norm{
\left(\sum_j\abs{x_j}^p\right)^{1/p}
}_p^p
=
\sum_j\norm{x_j}_p^p,
\]
which is used both in the lifting argument and in the bootstrap.
Thus the same proof extends, with the appropriate constants, to
Banach lattices for which the corresponding upper and lower
$\ell^r$-lattice estimates are simultaneously available, together
with the dual estimates on $X^*$.

It would be interesting to determine the optimal formulation of this
argument solely in terms of lattice convexity and concavity. This is
closely related to the framework of Deng, Lorist and Veraar
\cite[Section 4.1]{DLV}, who proved polynomial-logarithmic bounds for
positive strongly Kreiss bounded operators on $p$-convex and
$q$-concave Banach lattices.
\end{remark}

\begin{remark}
Theorem \ref{thmMain} answers
\cite[Question 5.8]{ArnoldCuny} affirmatively. It also implies that,
for a positive strongly Kreiss bounded operator on $L^p$,
the infimum of the exponents $\theta\geq0$ for which
\[
\norm{T^N}=O(N^\theta)
\]
is equal to zero, thereby answering \cite[Problem 5.6]{DLV} at the
level of the polynomial scale.
\end{remark}


\end{document}